\documentclass{article}
\usepackage[utf8]{inputenc}
\usepackage{amsfonts}
\usepackage{amsmath}
\usepackage{amssymb}
\usepackage{amsthm}
\usepackage{physics}
\usepackage{geometry}
\usepackage{listings}
\usepackage{xcolor}
\usepackage{hyperref}
\usepackage{nccmath}
\usepackage{multicol}
\usepackage{mathtools}
\usepackage{graphicx}
\usepackage{float}
\usepackage{subcaption}
\usepackage{tikz}
\usepackage{pgfplots}
\usepackage{mathrsfs}
\usepackage{faktor}
\graphicspath{ {./images/} }

\newcommand{\Z}{\mathbb{Z}}
\newcommand{\SL}{\text{SL}_2(\mathbb{Z})}

\newcommand{\oo}[1]{\overline{#1}}

\newcommand{\I}{\begin{psmallmatrix}
    1 & 0 \\
    0 & 1
\end{psmallmatrix}}

\newtheorem{theorem}{Theorem}[section]

\newtheorem{definition}[theorem]{Definition}
\newtheorem{lemma}[theorem]{Lemma}
\newtheorem{proposition}[theorem]{Proposition}
\newtheorem*{remark}{Remark}
\theoremstyle{definition}
\newtheorem{example}[theorem]{Example}

\title{Fast Computation of Generalized Dedekind Sums}
\author{\vspace{-1mm}\small Preston Tranbarger \\ \vspace{-2mm}\scriptsize\href{mailto:prestontranbarger@tamu.edu}{prestontranbarger@tamu.edu} \\ \vspace{-2mm}\scriptsize Department of Mathematics \\ \vspace{-2mm}\scriptsize Texas A\&M University \\ \vspace{-2mm}\scriptsize College Station, TX 77843-3368, U.S.A. \and \vspace{-1mm}\small Jessica Wang \\ \vspace{-2mm}\scriptsize \href{mailto:jwang22@wpi.edu}{jwang22@wpi.edu} \\
\vspace{-2mm}\scriptsize Department of Mathematical Sciences \\ \vspace{-2mm}\scriptsize Worcester Polytechnic Institute \\
\vspace{-2mm}\scriptsize Worcester, MA 01609-2280, U.S.A.}
\date{}

\begin{document}

\maketitle

\begin{abstract}
We construct an algorithm that reduces the complexity for computing generalized Dedekind sums from exponential to polynomial time. We do so by using an efficient word rewriting process in group theory.
\end{abstract}

\section{Introduction and Main Result}

The classical Dedekind sum is well-studied both inside and outside of number theory due to its connections with the Dedekind eta function and its applications in topology and combinatorial geometry. For more background on classical Dedekind sums, we refer the reader to \cite{rademacher-grosswald}.

Let $h$ and $k$ be coprime integers with $k>0.$ The classical Dedekind sum is defined as

    \[s(h,k)=\sum_{n=1}^k B_1\bigg(\cfrac{n}{k}\bigg)B_1\bigg(\cfrac{hn}{k}\bigg),\]
where $B_1(x)$ is the first Bernoulli function

\[B_1(x)=\begin{cases}
			0, & \text{if $x\in\Z$}\\
            x-\lfloor x\rfloor -\frac{1}{2}, & \text{otherwise.}
		 \end{cases}\]
The classical Dedekind sum satisfies the following reciprocity property (\cite{rademacher-grosswald}):
 \[s(h,k)=-s(k,h)+\frac{1}{12}\bigg(\frac{h}{k}+\frac{1}{hk}+\frac{k}{h}\bigg)-\frac{1}{4}.\]
Using the definition of classical Dedekind sum, it is readily seen that it can be computed in $O(k)$ time. However, one can obtain an $O(\log(k))$ time algorithm to compute the classical Dedekind sum using the reciprocity property (\cite{log-time-eucld}).

The generalized Dedekind sum associated with newform Eisenstein series was introduced by Stucker, Vennos, and Young in \cite{SVY}. Since then, various aspects of generalized Dedekind sums have been studied, including the kernel (\cite{kernel}, \cite{kernel2}), the image (\cite{mitchPaper}), and their general behaviour (\cite{average}).

\begin{definition}\label{SVYDEF}
Let $\gamma=\begin{psmallmatrix}
a&b\\c&d
\end{psmallmatrix}\in\Gamma_0(q_1q_2)$ with primitive Dirichlet characters $\chi_1,\chi_2$ and respective conductors $q_1,q_2.$ Let $q_1,q_2>1$ and $\chi_1\chi_2(-1)=1,$ then
  $$S_{\chi_1,\,\chi_2}\begin{pmatrix}
    a & b \\
    c & d
\end{pmatrix}=\sum_{j=1}^{c}\sum_{i=1}^{q_1}\left(\oo{\chi_2(j)\chi_1(i)}B_1\left(\frac{j}{c}\right)B_1\left(\frac{n}{q_1}+\frac{aj}{c}\right)\right).$$   
\end{definition}

The generalized Dedekind sum has the following crossed homomorphism property.
\begin{lemma}[Crossed Homomorphism Property \cite{SVY}]\label{lemma-cross-hom}
Let $\gamma_1,\gamma_2\in\Gamma_0(q_1q_2).$ Then
\[S_{\chi_1,\chi_2}(\gamma_1\gamma_2)=S_{\chi_1,\chi_2}(\gamma_1)+\psi(\gamma_1)S_{\chi_1,\chi_2}(\gamma_2),\]
where $\psi\begin{psmallmatrix}
    a&b\\c&d
\end{psmallmatrix}=\chi_1\oo{\chi_2}(d).$
\end{lemma}

\begin{remark}
In Lemma \ref{lemma-cross-hom}, $\psi(\gamma)$ is trivial on $\Gamma_1(q_1q_2),$ so $S_{\chi_1,\chi_2}$ may be viewed as an element of $\text{Hom}(\Gamma_1(q_1 q_2),\mathbb{C}).$
\end{remark}

Note that it requires $O(cq_1)$ time to compute a generalized Dedekind sum from the definition. Similar to the classical Dedekind sums, we are interested in constructing a faster algorithm. Instead of using the reciprocity property, we provide an alternative approach using a word-rewriting process.

\begin{theorem}\label{thm:main}
Given primitive Dirichlet characters $\chi_1,\chi_2$ and respective conductors $q_1,q_2>1$ such that $\chi_1\chi_2(-1)=1.$ Let $\gamma=\begin{psmallmatrix}
a&b\\c&d
\end{psmallmatrix}\in\Gamma_0(q_1q_2)$. For fixed $q_1,q_2$, the time complexity of finding $S_{\chi_1,\chi_2}(\gamma)$ as a function of $\gamma$ is $O(\log(c))$.
\end{theorem}
\begin{remark}
The algorithm for Theorem \ref{thm:main} can be found in Section \ref{sec:algo}. We give the specific details of our model of computation in Section \ref{section:algo-analysis}. 
\end{remark}


Before diving into the technicalities of the algorithm, we provide the reader with a general outline. Given $\gamma\in\Gamma_1(q_1q_2)<\SL$ written as a word in the generators of $\SL$, we can apply the Reidemeister rewriting process to express it as a word in the elements of a particular generating set of $\Gamma_1(q_1q_2)$. By precomputing the Dedekind sum of each element of this generating set, we can use Lemma \ref{lemma-cross-hom} to compute any Dedekind sum. However, in using the Reidemeister rewriting process, the length of the word can be exponentially large in terms of the logarithms of the entries of $\gamma.$ Therefore, to achieve the polynomial time in Theorem \ref{thm:main}, we modify the Reidemeister rewriting process, as in Theorem \ref{thm:modified-reidemeister} below, to collect the exponents of successive letters in the rewritten word. In the specific case of $\Gamma_1(q_1q_2),$ we develop a useful identity in Lemma \ref{lemma-T-cycle} to ensure a finite alphabet.

\section{Preliminaries}
\subsection{General Preliminaries}
In this section we will define some general group theoretic definitions and results which will aid in the construction of the algorithm. 
For the rest of this subsection, we let $G$ be a finitely generated group and $H$ be a subgroup of $G$. 
We will work with specific groups in Section \ref{section-specific}.

\begin{definition}
We say $\mathcal{T}$ is a \emph{right transversal} of $H$ in $G$ if each right coset of $H$ in $G$ contains exactly one element of $\mathcal{T}$. Moreover, $\mathcal{T}$ must contain the identity.
\end{definition}

Note that a transversal differs from an arbitrary set of coset representatives in that it must contain an identity. This fact proves to be essential for later preliminaries.


\begin{definition}\label{bar-function-def}
Given a right transversal $\mathcal{T}$ of $H$ in $G$, a \emph{right coset representative function} for $\mathcal{T}$ is a mapping: $G\to \mathcal{T}$ via
$g \mapsto \oo{g}$, where $\oo{g}$ is the unique element in $\mathcal{T}$ such that $Hg=H\oo{g}$.

\end{definition}

We present a simple lemma which will be used repeatedly throughout this paper.

\begin{lemma}\label{nested-cosets}
Given a right transversal of $H$ in $G$ and $a,\,b\in G$,
$$\oo{ab}=\oo{\oo{a}b}.$$
\end{lemma}
\begin{proof}
By Definition \ref{bar-function-def}, $H(\oo{ab})=H(ab)=(Ha)b=(H\oo{a})b=H(\oo{a}b)=H(\oo{\oo{a}b}).$
\end{proof}

We continue by defining an important function and exploring some of its properties.
\begin{definition}\label{Ufunction-def}
Given a right transversal of $H$ in $G$ and $a,b\in G$, we define \[U(a,b)=ab(\oo{ab})^{-1}.\]
\end{definition}

\begin{lemma}\label{Ufunction-in-H}
Given a right transversal of $H$ in $G$ and $a,b\in G$, then
$U(a,b)\in H.$
\end{lemma}
\begin{proof}
By Definition \ref{bar-function-def}, $Hab=H\oo{ab}$, thus $Hab(\oo{ab})^{-1}=H$, so $ab(\oo{ab})^{-1}\in H$.
\end{proof}

Given a finite set of generators for a group, we use the information thus far to describe a set of generators for a given subgroup.

\begin{lemma}[Schreier's Lemma {\cite[Theorem 2.7]{reidemeister}}]\label{Schreier} Let $\mathcal{S}$ be a set which finitely generates $G$, and let $\mathcal{T}$ be a right transversal of $H$ in $G$.  The set of Schreier generators
$$\{U(t,\,s)\,:\,t\in\mathcal{T},\,s\in\mathcal{S}\}$$
generates $H$.
\end{lemma}

\begin{remark}
We say a set generates a group if every element in the group can be expressed as a combination of elements in the set and their inverses.

\end{remark}

We now describe a rewriting process.

\begin{theorem}[Reidemeister Rewriting Process {\cite[Corollary 2.7.2]{reidemeister}}]\label{thm:Reidemeister-Rewriting}
Let $G=\langle g_1,\cdots, g_n\rangle.$ Let $h=g_{q_1}^{\epsilon_1}g_{q_2}^{\epsilon_2}\cdots g_{q_r}^{\epsilon_r}\in H$ (where $\epsilon_k=\pm 1$) be a word in the $g_i$. Fix a right transversal of $H$ in $G$. Define the mapping $\tau$ of the word $h$ by
\[\tau(h)=U(p_1,g_{q_1})^{\epsilon_1}U(p_2,g_{q_2})^{\epsilon_2}\cdots U(p_r,g_{q_r})^{\epsilon_r},\]
where 
\[p_k = \begin{cases}
    \oo{g_{q_1}^{\epsilon_1}g_{q_2}^{\epsilon_2}\cdots g_{q_{k-1}}^{\epsilon_{k-1}}}  & \text{if } \epsilon_k=1 \\
    \oo{g_{q_1}^{\epsilon_1}g_{q_2}^{\epsilon_2}\cdots g_{q_k}^{\epsilon_k}}  & \text{if } \epsilon_k=-1.
\end{cases}\]
Then $\tau(h)=h,$ for all $h\in H.$
\end{theorem}


The Reidemeister rewriting process allows us to express a word in the generators of $G$ as a word in the Schreier generators of $H$ (using Lemma \ref{Schreier}). 

\begin{example}\label{ex:long-terms}
Let $G=\langle g_1,\cdots, g_n\rangle$, and $H$ be a subgroup of $G$. Let $h=g_1g_1g_1g_2^{-1}g_2^{-1}\in H$, then by Theorem \ref{thm:Reidemeister-Rewriting}, 

\begin{equation}\label{eq:reidemeister}
    \tau(h)=U\big(\oo{1},g_1\big)U\big(\oo{g_1},g_1\big)U\big(\oo{g_1^2},g_1\big)U\big(\oo{g_1^3g_2^{-1}},g_2\big)^{-1}U\big(\oo{g_1^3g_2^{-2}},g_2\big)^{-1}=h.
\end{equation}
\end{example}

Note that Theorem \ref{thm:Reidemeister-Rewriting} requires $\epsilon_k=\pm 1$ (for example, $h$ must be written as $g_1g_1g_1g_2^{-1}g_2^{-1},$ not $g_1^3g_2^{-2}$). Since the length of $\tau(h)$ is the same as the length of $h$, this rewriting process is often inefficient. We provide Theorem \ref{thm:modified-reidemeister} to reduce the number of computations, which requires the lemma below.


\begin{lemma}\label{lemma:exponent}
Let $a,b\in G$ and $k\in\mathbb{Z}_{>0}$. Given a right transversal of $H$ in $G$, the following product identities hold:
\begin{align}
    U\big(\oo{a},\,b^k\big)&=U\big(\oo{a},\,b\big)U\big(\oo{ab},\,b\big)\ldots U\big(\oo{ab^{k-1}},\,b\big) \label{eq-exponents},\\ 
    U\big(\oo{a},\,b^{-k}\big)&=U\big(\oo{ab^{-1}},\,b\big)^{-1}U\big(\oo{ab^{-2}},\,b\big)^{-1}\ldots U\big(\oo{ab^{-k}},\,b\big)^{-1}.\label{eq-inverse-exp}
\end{align}
\end{lemma}
\begin{proof}
By Definition \ref{Ufunction-def} the right hand side of (\ref{eq-exponents}) equals
\begin{equation}\label{eq-rhs}
    \oo{a}b\left(\oo{\oo{a}b}\right)^{-1}\oo{ab}b\left(\oo{\oo{ab}b}\right)^{-1}\ldots\oo{ab^{k-1}}b\left(\oo{\oo{ab^{k-1}}b}\right)^{-1}.
\end{equation}
Applying Lemma \ref{nested-cosets}, this simplies to
\begin{equation}\label{whentheumwhenthethewhenthe}
\oo{a}b\big(\oo{ab}\big)^{-1}\oo{ab}b\big(\oo{ab^2}\big)^{-1}\ldots\oo{ab^{k-1}}b\left(\oo{\oo{a}b^k}\right)^{-1}.
\end{equation}
Note that many terms cancel, so (\ref{whentheumwhenthethewhenthe}) becomes
$\oo{a}b^k\big(\oo{\oo{a}b^k}\big)^{-1}=U\big(\oo{a},\,b^k\big).$
The proof of (\ref{eq-inverse-exp}) follows in a similar manner, though care is required in handling the inverses. 
\end{proof}



\begin{theorem}[Modified Reidemeister Rewriting Process]\label{thm:modified-reidemeister}
Given a right transversal of $H$ in $G$, let $G=\langle g_1,\cdots, g_n\rangle$. Let $h=g_{q_1}^{a_1}g_{q_2}^{a_2}\cdots g_{q_r}^{a_r}\in H$ (where $a_i\in\Z_{\neq 0}$) be a word in powers of the $g_i$. Define the mapping $\tau$ of the word $h$ by
\[\tau(h)=U(p_1,g_{q_1}^{a_1})U(p_2,g_{q_2}^{a_2})\cdots U(p_r,g_{q_r}^{a_r}),\]
where 
\[p_k = 
    \oo{g_{q_1}^{a_1}g_{q_2}^{a_2}\cdots g_{q_{k-1}}^{a_{k-1}}}. 
\]
Then $\tau(h)=h,$ for all $h\in H.$
\end{theorem}
\begin{proof}
This follows by applying Lemma \ref{lemma:exponent} to Theorem \ref{thm:Reidemeister-Rewriting}.
\end{proof}

We illustrate Theorem \ref{thm:modified-reidemeister} and its proof by an example.

\begin{example}
Continuing with the assumptions in Example \ref{ex:long-terms}, we now write 
$h=g_1g_1g_1g_2^{-1}g_2^{-1}=g_1^3g_2^{-2}\in H$. We want to show (as Theorem \ref{thm:modified-reidemeister} claims) that $\tau(h)=U(\oo{1},g_1^3)U(\oo{g_1^3},g_2^{-2}).$
From Lemma \ref{lemma:exponent}, we have
\begin{align*}
    U(\oo{1},g_1^3)&=U(\oo{1},g_1)U(\oo{g_1},g_1)U(\oo{g_1^2},g_1),\\
    U(\oo{g_1^3},g_2^{-2})&=U(\oo{g_1^3g_2^{-1}},g_2)^{-1}U(\oo{g_1^3g_2^{-2}},g_2)^{-1}.
\end{align*}
So (\ref{eq:reidemeister}) becomes 
\[\tau(h)=U(\oo{1},g_1^3)U(\oo{g_1^3},g_2^{-2})=h,\]
which is in the form of Theorem \ref{thm:modified-reidemeister}.
\end{example}
As desired, this process provides us with a product expansion with far fewer terms than that of Theorem \ref{thm:Reidemeister-Rewriting}.





\subsection{Specific Preliminaries}\label{section-specific}

Let us now consider the subgroup $\Gamma_1(N)$ of $\SL$.

\begin{definition}
Let
$$S=\begin{pmatrix}
    0 & -1 \\
    1 &  0
\end{pmatrix},\; T=\begin{pmatrix}
    1 & 1 \\
    0 & 1
\end{pmatrix}.$$
\end{definition}

The following lemma is well known.

\begin{lemma}[{\cite[Theorem 1.1]{iwaniectopics}}]\label{TS-Decomposition} We have
$$\SL=\langle S,\,T\rangle.$$
More specifically, any matrix $M\in\SL$ can be decomposed into the following form:
\begin{align}\label{eq:TS-decomp}
    M=\begin{pmatrix}
    a & b \\
    c & d
\end{pmatrix}=\pm T^{a_1}ST^{a_2}S\ldots T^{a_{r-1}}ST^{a_r}.
\end{align}

Note that $-I=S^2$.
\end{lemma}
\begin{remark}
One can quickly compute the values of the $a_i$ via a variant on the Euclidean algorithm.
\end{remark}
\begin{remark}
In (\ref{eq:TS-decomp}), the sum of the $\abs{a_i}$ grows as $O(c)$ but $r$ grows as $O(\log(c))$ (this follows as a consequence of the above remark and {\cite[Section 4.5.3]{knuth}}).  
This is significant because the original Reidemeister rewriting process treats $M$ as a word in $S$ and $T$ with length $(\abs{a_1}+\cdots +\abs{a_r})+r=O(c)$. The modified Reidemeister rewriting process treats $M$ as a word in the powers of $S$ and $T$ with length $2r$ which grows as $O(\log(c))$.

\end{remark}

In the case of $\Gamma_1(N)$ in $\SL$, the modified Reidemeister rewriting process allows us to express a word in the generators of $\SL$ as a word in the elements of $\Gamma_1(N)$ with the form $U(l,g^k)$, where $l$ is in a right transversal of $\Gamma_1(N)$ in $\SL$ and $g\in\{S,T\}$ and $k\in\mathbb{Z}$. However, since $k$ is arbitrarily large, this set of elements is infinite. We present a way to reduce this infinite set to a finite set.


\begin{lemma}\label{coset-T-cycle}
Given a right transversal of $\Gamma_1(N)$ in $\SL$,
$$\oo{MT^N}=\oo{M}$$
for all $M\in \SL$.
\end{lemma}
\begin{proof}
Note $T^N\in\Gamma(N)$. Since $\Gamma(N)$ is normal in $\SL$, we have  $MT^NM^{-1}\in\Gamma(N)<\Gamma_1(N)$ for all $M\in\SL$. Thus $\Gamma_1(N)MT^N=\Gamma_1(N)M$ and $\oo{MT^N}=\oo{M}$.
\end{proof}

\begin{lemma}\label{lemma-T-cycle}
Let $a=qN+r$ for $0\leq r<N$ and let $M\in\SL$. Given a right transversal of $\Gamma_1(N)$ in $\SL$,
$$U\big(\oo{M},\,T^a\big)=U^q\big(\oo{M},T^N\big)U\big(\oo{M},\,T^r\big).$$
\end{lemma}
\begin{proof}
Beginning with Definition \ref{Ufunction-def}, we have 
\begin{equation}\label{eq:reduce-power}
    U\big(\oo{M},\,T^a\big)=\oo{M}T^a\left(\oo{\oo{M}T^a}\right)^{-1}.
\end{equation}
By applying Lemma \ref{coset-T-cycle}, expanding $T^a$, and multiplying by the identity $I=(\oo{M})^{-1}\oo{M}$, (\ref{eq:reduce-power}) becomes
\[\oo{M}T^{qN}\big(\oo{M}\big)^{-1}\oo{M}T^r\big(\oo{\oo{M}T^r}\big)^{-1}.\]
By Lemmas \ref{nested-cosets} and \ref{coset-T-cycle}, $\oo{M}=\oo{\oo{M}}=\oo{\oo{M}T^N}$. Since $\oo{M}T^{qN}(\oo{M})^{-1}=\big(\oo{M}T^N(\oo{M})^{-1}\big)^q$, we get that (\ref{eq:reduce-power}) equals
$$\left(\oo{M}T^N\left(\oo{\oo{M}T^N}\right)^{-1}\right)^q\oo{M}T^r\left(\oo{\oo{M}T^r}\right)^{-1}.\qedhere$$
\end{proof}

Using Lemma \ref{lemma-T-cycle} in conjunction with Theorem \ref{thm:modified-reidemeister}, we can rewrite every word in the generators of $\SL$ as a word in the letters of a finite subset of $\Gamma_1(N)$. This is central to our algorithm in Section \ref{sec:algo}. However, before we are able to present this algorithm, we need a few more results describing the structure of congruence subgroups of $\SL$. One can find these results in many sources, including \cite{stein}.

\begin{lemma}\label{lemma:index} We have
\[[\Gamma_1(N):\Gamma_0(N)]=N\prod_{p|N}\left(1-\frac{1}{p}\right),\]
\[[\Gamma_0(N):\SL]=N\prod_{p|N}\left(1+\frac{1}{p}\right),\]
\[[\Gamma_1(N):\SL]=N^2\prod_{p|N}\left(1-\frac{1}{p^2}\right).\]
\end{lemma}

\begin{lemma}\label{cosets-g1-in-g0}
There exists a bijection $\Gamma_1(N)\backslash\Gamma_0(N)\to (\mathbb{Z}/N\mathbb{Z})^{\cross}$ via 
\[\begin{pmatrix}
    a & b \\
    c & d
\end{pmatrix}\mapsto d\bmod N.\]
\end{lemma}

\begin{lemma}[{\cite[Proposition 8.6]{stein}}]\label{cosets-g1-in-sl2z}
Let $P=\{(c,d):c,d\in\Z/N\Z,\,\gcd(c,d,N)=1\}.$
There exists a bijection $\Gamma_1(N)\backslash\SL\to P$ via
$$\begin{pmatrix}
    a & b \\
    c & d
\end{pmatrix}\mapsto (c\bmod N,d\bmod N).$$
\end{lemma}

\section{Algorithm}
In this section, we provide an algorithm for computing the generalized Dedekind sums and a time complexity analysis for each step of the algorithm. We also compare our algorithm to the naive algorithm of simply using Definition \ref{SVYDEF}. We divide our algorithm into precomputations (which only needs to be computed once for each pair of characters) and the main computation.

\subsection{Stating the Algorithm}\label{sec:algo}
Let $N=q_1q_2$ for primitive Dirichlet characters $\chi_1,\chi_2$ with respective conductors $q_1,q_2.$ Let $q_1,q_2>1$ and $\chi_1\chi_2(-1)=1.$ Given $\gamma_0=\begin{psmallmatrix}
a&b\\c&d
\end{psmallmatrix}\in\Gamma_0(N),$ we present an algorithm to find $S_{\chi_1,\chi_2}(\gamma_0).$ 

\subsubsection*{Group Theoretic Precomputation}\label{group precomp}
\begin{itemize}
    \item Find a right transversal $\mathcal{T}_{\Gamma_0}$ of $\Gamma_1(N)$ in $\Gamma_0(N)$ (using Lemma \ref{cosets-g1-in-g0}).
    \item Find a right transversal $\mathcal{T}_{\SL}$ of $\Gamma_1(N)$ in $\SL$ (using Lemma \ref{cosets-g1-in-sl2z}).
    \item Find the set $\mathcal{U}=\{U(t,T^i):t\in\mathcal{T}_{\SL},1\leq i\leq N\} \cup \{U(t,S^k):t\in\mathcal{T}_{\SL},0\leq k\leq 2\}$. Note this set includes the set of the Schreier generators of $\Gamma_1(N)$ in $\SL$ (see Lemma \ref{Schreier}).
\end{itemize}

\subsubsection*{Dedekind Sum Precomputation}
\begin{itemize}
    \item Use Definition \ref{SVYDEF} to compute the Dedekind sums $S_{\chi_1,\chi_2}\big(\mathcal{T}_{\Gamma_0}\big).$
    \item Use Definition \ref{SVYDEF} to compute the Dedekind sums $S_{\chi_1,\chi_2}(\mathcal{U})$.
\end{itemize}

\subsubsection*{The Main Computation}
We write $\gamma_0=\gamma_1g,$ where $\gamma_1\in\Gamma_1(N)$ and $g\in\mathcal{T}_{\Gamma_0}.$ Let $g\mapsto\oo{g}$ denote the right coset representative function uniquely described by $\mathcal{T}_{\SL}$ per Definition \ref{bar-function-def}. 
By Lemma \ref{lemma-cross-hom},
$$S_{\chi_1,\chi_2}(\gamma_0)=S_{\chi_1,\chi_2}(\gamma_1)+S_{\chi_1,\chi_2}(g).$$
Since $g\in\mathcal{T}_{\Gamma_0}$, $S_{\chi_1,\chi_2}(g)$ has been precomputed, so we are now only concerned with $S_{\chi_1,\chi_2}(\gamma_1).$
Using Lemma \ref{TS-Decomposition}, we write \[\gamma_1=\pm T^{a_1}ST^{a_2}S\ldots T^{a_{r-1}}ST^{a_r}.\] Using Theorem \ref{thm:modified-reidemeister}, we rewrite 
\begin{equation}\label{maincomp1}
\tau(\gamma_1)=U(\oo{p_1},T^{a_1})U(\oo{p_1T^{a_1}},S)U(\oo{p_2},T^{a_2})U(\oo{p_2T^{a_2}},S)\cdots U(\oo{p_r},T^{a_r})U(\oo{p_rT^{a_r}},\pm I)=\gamma_1,
\end{equation}
where
$$p_k=T^{a_1}ST^{a_2}S\ldots T^{a_{k-1}}S.$$
Now we apply Lemma \ref{lemma-T-cycle}. For each exponent of $T$, we write $a_i=q_iN+r_i$ with $0\leq r_i<N$. Then
\begin{equation}\label{maincomp2}
U\big(\oo{p_i},T^{a_i}\big)=U^{q_i}\big(\oo{p_i},T^N\big)U\big(\oo{p_i},\,T^{r_i}\big).
\end{equation}
We apply the Dedekind sum to (\ref{maincomp2}). Since $\mathcal{U}\subset\Gamma_1(N)$, by Lemma \ref{lemma-cross-hom},
\begin{equation}\label{eq:sum}
    S_{\chi_1,\chi_2}\left(U\big(\oo{p_i},\,T^a\big)\right)=q_iS_{\chi_1,\chi_2}\left(U\big(\oo{p_i},T^N\big)\right)+S_{\chi_1,\chi_2}\left(U\big(\oo{p_i},T^{r_i}\big)\right).
\end{equation}

Using ($\ref{maincomp1}$) and ($\ref{maincomp2}$), we can express $\gamma_1$ as a product of elements in $\mathcal{U}$. Applying Lemma \ref{lemma-cross-hom} to $S_{\chi_1,\chi_2}(\tau(\gamma_1))$ and expanding via (\ref{eq:sum}), we acquire the desired Dedekind sum, each term of which has been precomputed.

\begin{remark}
Given $\mathcal{T}_{\SL}$ and $g\in\SL$, we can use Proposition \ref{cosets-g1-in-sl2z} to determine $\oo{g}$.
\end{remark}

\subsubsection*{Example Computation}
Consider $\Gamma_0(9).$ Let $\chi_1=\chi_2$ be the primitive character modulo $3$ with conductors $q_1=q_2=3.$
We want to compute $S_{\chi_1,\chi_2}(\gamma_0)$ where
$$\gamma_0=\begin{pmatrix}
    17 & 32 \\
    9 & 17
\end{pmatrix}.$$
From the precomputations, we acquire
a right transversal of $\Gamma_1(9)$ in $\Gamma_0(9)$
$$\mathcal{T}_{\Gamma_0}=\left\{\I,\begin{psmallmatrix}
    5 & 1 \\
    9 & 2
\end{psmallmatrix},\begin{psmallmatrix}
    7 & 3 \\
    9 & 4
\end{psmallmatrix},\begin{psmallmatrix}
    2 & 1 \\
    9 & 5
\end{psmallmatrix},\begin{psmallmatrix}
    4 & 3 \\
    9 & 7
\end{psmallmatrix},\begin{psmallmatrix}
    8 & 7 \\
    9 & 8
\end{psmallmatrix}\right\},$$
a right transversal of $\Gamma_1(9)$ in $\SL$
\begin{align*}
    \mathcal{T}_{\SL}=\big\{&\I,\begin{psmallmatrix}
        5 & 1 \\
        9 & 2
    \end{psmallmatrix},\begin{psmallmatrix}
        7 & 3 \\
        9 & 4
    \end{psmallmatrix},\begin{psmallmatrix}
        2 & 1 \\
        9 & 5
    \end{psmallmatrix},\begin{psmallmatrix}
        4 & 3 \\
        9 & 7
    \end{psmallmatrix},
    \cdots ,\begin{psmallmatrix}
        5 & 8 \\
        8 & 13
    \end{psmallmatrix},\begin{psmallmatrix}
        5 & 3 \\
        8 & 5
    \end{psmallmatrix},\begin{psmallmatrix}
        7 & 13 \\
        8 & 15
    \end{psmallmatrix},\begin{psmallmatrix}
        7 & 6 \\
        8 & 7
    \end{psmallmatrix},\begin{psmallmatrix}
        1 & 2 \\
        8 & 17
    \end{psmallmatrix}\big\},
\end{align*}
and the set
\[\mathcal{U}=\{U(t,T^i):t\in\mathcal{T}_{\SL},1\leq i\leq 9\} \cup \{U(t,S^k):t\in\mathcal{T}_{\SL},0\leq k\leq 2\}.\]
From the precomputations, we also acquire the Dedekind sums $S_{\chi_1,\chi_2}(\mathcal{T}_{\Gamma_0})$ and $S_{\chi_1,\chi_2}(\mathcal{U}).$
By Lemma \ref{cosets-g1-in-g0}, we can write $\gamma_0=\gamma_1 g,$ where
$$\gamma_1=\begin{psmallmatrix}
        -152 & 137 \\
        -81 & 73
    \end{psmallmatrix}\in\Gamma_1(9)\qquad\text{and}\qquad g=\begin{psmallmatrix}
        8 & 7 \\
        9 & 8
    \end{psmallmatrix}\in\mathcal{T}_{\Gamma_0}.\;$$
Since $g\in\mathcal{T}_{\Gamma_0}$, $S_{\chi_1,\chi_2}(g)$ has been precomputed, so now we only need to compute $S_{\chi_1,\chi_2}(\gamma_1).$
Using Lemma \ref{TS-Decomposition}, we compute
\begin{equation}\label{eq:example-TS}
    \gamma_1=-T^{1}ST^{-2}ST^{-2}ST^{-2}ST^{-2}ST^{-2}ST^{-2}ST^{-2}ST^{-11}ST^{-1}.
\end{equation}

Applying Theorem \ref{thm:modified-reidemeister} to 
(\ref{eq:example-TS}) with all $p_i$ written in matrix forms, we get
\begin{align*}
\tau(\gamma_1)=&U\big(\oo{\begin{psmallmatrix}
    1 & 0 \\
    0 & 1
\end{psmallmatrix}},T^1\big)U\big(\oo{\begin{psmallmatrix}
    1 & 1 \\
    0 & 1
\end{psmallmatrix}},S\big)U\big(\oo{\begin{psmallmatrix}
    1 & -1 \\
    1 & 0
\end{psmallmatrix}},T^{-2}\big)U\big(\oo{\begin{psmallmatrix}
    1 & -3 \\
    1 & -2
\end{psmallmatrix}},S\big)
    \cdots \\
    &\cdots U\big(\oo{\begin{psmallmatrix}
        -15 & -13 \\
        -8 & -7
    \end{psmallmatrix}},T^{-11}\big)U\big(\oo{\begin{psmallmatrix}
        -15 & 152 \\
        -8 & 81
    \end{psmallmatrix}},S\big)U\big(\oo{\begin{psmallmatrix}
        152 & 15 \\
        81 & 8
    \end{psmallmatrix}},T^{-1}\big)U\big(\oo{\begin{psmallmatrix}
        152 & -137 \\
        81 & -73
    \end{psmallmatrix}},-I\big)=\gamma_1.
\end{align*}
Applying Lemmas \ref{lemma-T-cycle} and \ref{cosets-g1-in-sl2z} to each term of the above product, we get the following computation.

\begin{center}
\begin{tabular}{l l}
    $U\big(\oo{\begin{psmallmatrix}
        1 & 0 \\
        0 & 1
    \end{psmallmatrix}},T^1\big)=U^0(\begin{psmallmatrix}
        1 & 0 \\
        0 & 1
    \end{psmallmatrix},T^9)U(\begin{psmallmatrix}
        1 & 0 \\
        0 & 1
    \end{psmallmatrix},T^1)$, & $U\big(\oo{\begin{psmallmatrix}
        1 & 1 \\
        0 & 1
    \end{psmallmatrix}},S\big)=U(\begin{psmallmatrix}
        1 & 0 \\
        0 & 1
    \end{psmallmatrix}, S)$, \\
    $U\big(\oo{\begin{psmallmatrix}
        1 & -1 \\
        1 & 0
    \end{psmallmatrix}},T^{-2}\big)=U^{-1}(\begin{psmallmatrix}
        1 & 8 \\
        1 & 9
    \end{psmallmatrix},T^9)U(\begin{psmallmatrix}
        1 & 8 \\
        1 & 9
    \end{psmallmatrix},T^7)$, & $U\big(\oo{\begin{psmallmatrix}
        1 & -3 \\
        1 & -2
    \end{psmallmatrix}},S\big)=U(\begin{psmallmatrix}
        1 & 6 \\
        1 & 7
    \end{psmallmatrix},S)$, \\
    
    \qquad\qquad\qquad\qquad\qquad\qquad\qquad\qquad\qquad\qquad\qquad\vdots\\
    $U\big(\oo{\begin{psmallmatrix}
        -15 & -13 \\
        -8 & -7
    \end{psmallmatrix}},T^{-11}\big)=U^{-2}(\begin{psmallmatrix}
        1 & 1 \\
        1 & 2
    \end{psmallmatrix},T^9)U(\begin{psmallmatrix}
        1 & 1 \\
        1 & 2
    \end{psmallmatrix},T^7)$, & $U\big(\oo{\begin{psmallmatrix}
        -15 & 152 \\
        -8 & 81
    \end{psmallmatrix}},S\big)=U(\begin{psmallmatrix}
        1 & 8 \\
        1 & 9
    \end{psmallmatrix},S)$, \\
    $U\big(\oo{\begin{psmallmatrix}
        152 & 15 \\
        81 & 8
    \end{psmallmatrix}},T^{-1}\big)=U^{-1}(\begin{psmallmatrix}
        8 & 7 \\
        9 & 8
    \end{psmallmatrix},T^9)U(\begin{psmallmatrix}
        8 & 7 \\
        9 & 8
    \end{psmallmatrix},T^8)$, & $U\big(\oo{\begin{psmallmatrix}
        152 & -137 \\
        81 & -73
    \end{psmallmatrix}},-I\big)=U(\begin{psmallmatrix}
        8 & 7 \\
        9 & 8
    \end{psmallmatrix},S^2).$
\end{tabular}
\end{center}

Note that every term on the right hand side of these equalities are in the precomputed set $\mathcal{U}.$ Thus, using Lemma \ref{lemma-cross-hom}, 
\begin{align*}
    S_{\chi_1,\chi_2}(\gamma_1)=0\cdot S_{\chi_1,\chi_2}(U(\begin{psmallmatrix}
        1 & 0 \\
        0 & 1
    \end{psmallmatrix},T^9))&+S_{\chi_1,\chi_2}(U(\begin{psmallmatrix}
        1 & 0 \\
        0 & 1
    \end{psmallmatrix},T^1))+S_{\chi_1,\chi_2}(U(\begin{psmallmatrix}
        1 & 0 \\
        0 & 1
    \end{psmallmatrix},S)) \\
    -1\cdot S_{\chi_1,\chi_2}(U(\begin{psmallmatrix}
        1 & 8 \\
        1 & 9
    \end{psmallmatrix},T^9))&+S_{\chi_1,\chi_2}(U(\begin{psmallmatrix}
        1 & 8 \\
        1 & 9
    \end{psmallmatrix},T^7))+S_{\chi_1,\chi_2}(U(\begin{psmallmatrix}
        1 & 6 \\
        1 & 7
    \end{psmallmatrix},S)) \\
    &\qquad\qquad\quad\;\,\vdots \\
    -2\cdot S_{\chi_1,\chi_2}(U(\begin{psmallmatrix}
        1 & 1 \\
        1 & 2
    \end{psmallmatrix},T^9))&+S_{\chi_1,\chi_2}(U(\begin{psmallmatrix}
        1 & 1 \\
        1 & 2
    \end{psmallmatrix},T^7))+S_{\chi_1,\chi_2}(U(\begin{psmallmatrix}
        1 & 8 \\
        1 & 9
    \end{psmallmatrix},S)) \\
    -1\cdot S_{\chi_1,\chi_2}(U(\begin{psmallmatrix}
        8 & 7 \\
        9 & 8
    \end{psmallmatrix},T^9))&+S_{\chi_1,\chi_2}(U(\begin{psmallmatrix}
        8 & 7 \\
        9 & 8
    \end{psmallmatrix},T^8))+S_{\chi_1,\chi_2}(U(\begin{psmallmatrix}
        8 & 7 \\
        9 & 8
    \end{psmallmatrix},S^2)).
\end{align*}
Now, using the precomputed Dedekind sums,
\(S_{\chi_1,\chi_2}(\gamma_0)=S_{\chi_1,\chi_2}(\gamma_1)+S_{\chi_1,\chi_2}(g)=0.\)
\begin{remark}
This example shows that $\gamma_0=\begin{psmallmatrix}
    17 & 32 \\
    9 & 17
\end{psmallmatrix}$ lies in the kernel of this Dedekind sum. For more information on the kernel of Dedekind sums, we refer the reader to \cite{kernel, kernel2}.
\end{remark}

\subsection{Analysis of Algorithm}\label{section:algo-analysis}

First we discuss a simplified model of our computation. We consider a matrix multiplication as one operation. Since we work in $\SL$, we assume computing the inverse of a matrix takes constant time. Since the Dedekind sums lie in cyclotomic extensions of $\mathbb{Q}$, we can represent any Dedekind sum as a linear combination of powers of a given root of unity. Adding and multiplying these linear combinations is rather trivial, so we assume that addition and multiplication of Dedekind sums take constant time.

\begin{lemma}\label{constant-coset}
Given $g\in\SL$ and the right transversals $\mathcal{T}_{\SL}$ or $\mathcal{T}_{\Gamma_0}$, finding $\oo{g}$ under the right coset representative function for $\Gamma_1(N)\backslash\Gamma_0(N)$ or $\Gamma_1(N)\backslash\SL$ requires $O(1)$ time.
\end{lemma}
\begin{proof}
We simply use the bijections stated in Lemmas \ref{cosets-g1-in-g0} and \ref{cosets-g1-in-sl2z} to find the transversals.
\end{proof}

\begin{proposition}
In Section \ref{sec:algo}, the time complexity of the group theoretic precomputations is $O(N^3).$
\end{proposition}

\begin{proof}

The most computationally expensive step is finding the set $\mathcal{U}=\{U(t,T^i):t\in\mathcal{T}_{\SL},1\leq i\leq N\} \cup \{U(t,S^k):t\in\mathcal{T}_{\SL},0\leq k\leq 2\}.$ This set has at most $(N+3)|\mathcal{T}_{\SL}|$ elements, thus $|\mathcal{U}|$ grows as $O(N^3).$ Since we define matrix multiplication as one operation and determining the transversal of an element under the right coset representative function takes $O(1)$ time (by Lemma \ref{constant-coset}), it follows that the total time to compute the elements of $\mathcal{U}$ takes $O(N^3)$ time.\qedhere

\end{proof}

\begin{proposition}
Given $\gamma\in\Gamma_0(N)$, the time complexity of the Dedekind sum  precomputations for finding $S_{\chi_1,\chi_2}(\mathcal{T}_{\Gamma_0}\cup\mathcal{U})$ is $O(N^3Cq_1),$ where $C$ denotes the maximum absolute value of the lower-left entry of the elements in $\mathcal{T}_{\Gamma_0}\cup \mathcal{U}.$
\end{proposition}
Note that $C$ is solely dependent on the value of $N$.
\begin{proof}
Using Lemma \ref{lemma:index}, we have $$[\Gamma_1(N):\Gamma_0(N)]+|\mathcal{U}|=N\Big(\prod_{p|N}\big(1-\frac{1}{p}\big)\Big)+N^2(N+3)\Big(\prod_{p|N}\big(1-\frac{1}{p^2}\big)\Big)=O(N^3).$$
Since it takes $O(cq_1)$ steps to compute the Dedekind sum of a matrix with lower-left entry $c$ from Definition \ref{SVYDEF}, it takes $O(N^3Cq_1)$ steps to compute the Dedekind sum of all elements in $\mathcal{T}_{\Gamma_0}\cup\mathcal{U}.$ 
\end{proof}

\begin{definition}
Let $\gamma=\pm T^{a_1}ST^{a_2}S\cdots T^{a_k}\in\SL$. We say $T^{a_i}$ or $S$ is a letter, and $2k$ is the length of $\gamma$.
\end{definition}

Recall that Theorem \ref{thm:main} states that the main computation algorithm provided in Section \ref{sec:algo} has time complexity $O(\log(c)).$ We provide a proof below.

\begin{proof}[Proof of Theorem \ref{thm:main}] \label{proof-of-main}

We claim that the time complexity for the modified Reidemeister rewriting process (Theorem \ref{thm:modified-reidemeister}) is $O(\log(c)).$ 
In (\ref{maincomp1}), given $p_i$, each $p_{i+1}$ takes one operation to compute (through multiplying by a matrix on the right). We know that $2k$ grows in $O(\log(c))$ due to Lemma \ref{TS-Decomposition}, hence it takes $O(\log(c))$ steps to compute all $p_i$.
By Lemma \ref{constant-coset} given $g\in\SL,$ finding the corresponding element $\oo{g}$ in $\mathcal{T}_{\Gamma_0}$ has complexity $O(1)$. Hence in (\ref{maincomp1}), each of the $U$-functions takes $O(1)$ time to compute. Since $\gamma_1$ has length $2k$, it also takes $O(\log(c))$ steps to compute the $U$-functions given $p_i.$ Hence the time complexity for the modified Reidemeister rewriting process is $O(\log(c))+O(\log(c))=O(\log(c)).$

   
 Note that reducing the power of a generator in (\ref{maincomp2}) using Lemma \ref{lemma-T-cycle} and looking up the precomputed Dedekind sum of each letter both have complexity $O(1)$, which does not affect the time complexity of the algorithm.
\end{proof}
\subsection{Comparison of Algorithms}
In this section, we give some experimental evidence for the speed of our algorithm in comparison to that which uses Definition \ref{SVYDEF} (using the implementations found in Section \ref{code})
\begin{example}

Consider $\Gamma_0(28)$. Let $\chi_1$ be the primitive Dirichlet character with conductor $q_1=4$, and let $\chi_2$ be the primitive Dirichlet character with conductor $q_2=7$ such that $\chi_2(3)=\exp(2\pi i(5/6))$. We let $\gamma=\begin{psmallmatrix}a&b\\c&d\end{psmallmatrix}$ where $c=28k$, $0<a<c$, and $\gcd(a,c)=1.$ We choose $b$ and $d$ such that the exponent $a_r$ is $0$ after applying Lemma \ref{TS-Decomposition}. We compute the Dedekind sum $S_{\chi_1,\chi_2}(\gamma)$ of all matrices that satisfy the conditions, and graph the logarithm of the average time it takes to compute each $k$.

\begin{center}
\includegraphics[scale = 0.7]{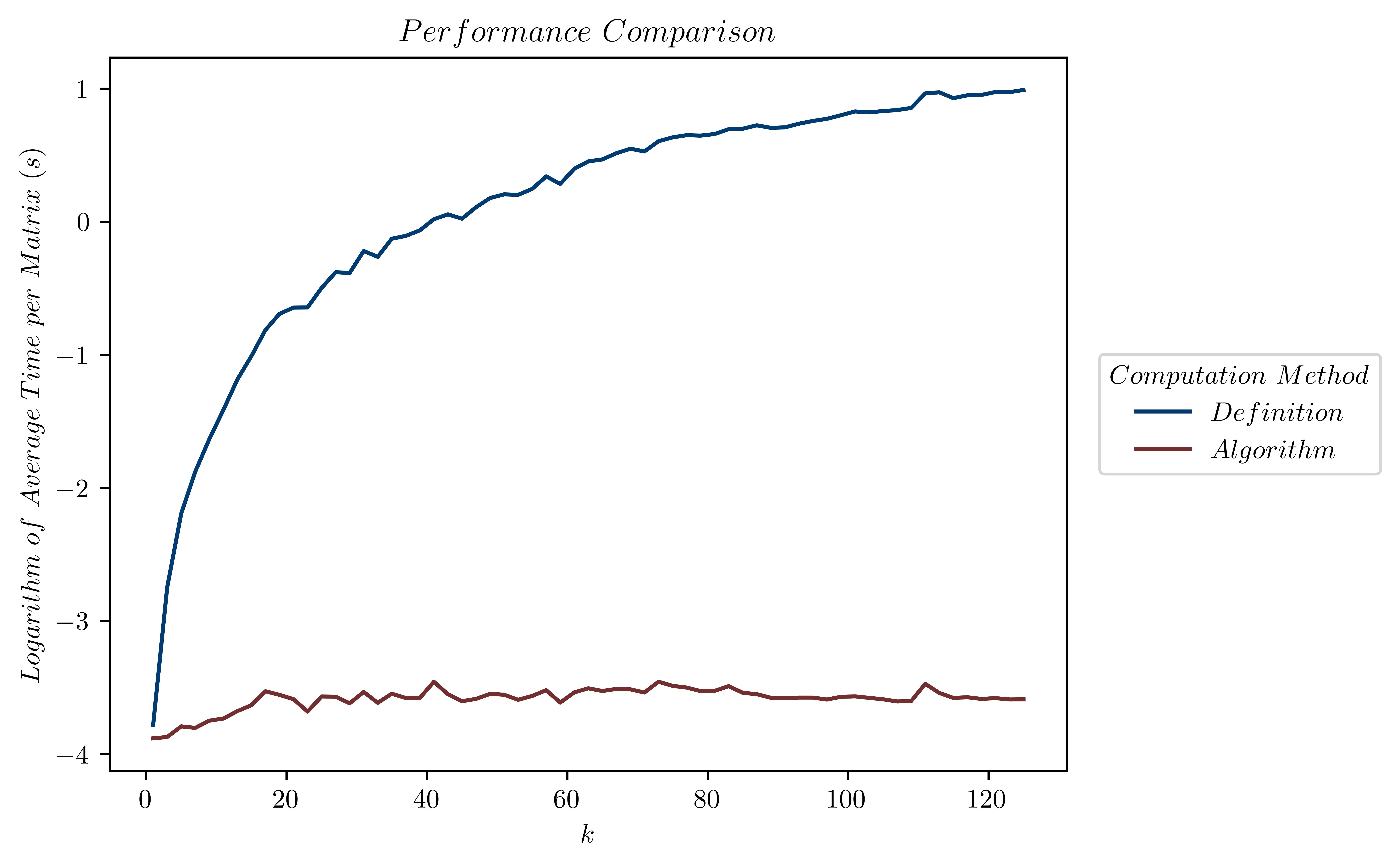}
\end{center}

 Note that the performance of the algorithm always exceeded that of the definition for this pair of characters.
\end{example}

\begin{example}
Now we present an example for a large matrix. Consider $\Gamma_0(35).$ Let $\chi_1$ be the primitive Dirichlet character with conductor $q_1=5$ such that $\chi_1(2)=-i$, and let $\chi_2$ be the primitive Dirichlet character with conductor $q_2=7$ such that $\chi_2(3)=\exp(2\pi i(1/3))$. Let $\gamma=\begin{pmatrix}
    46741638 & 43234369 \\
    43234205 & 39990117
\end{pmatrix}$.
Computing $S_{\chi_1,\chi_2}(\gamma)$ by Definition \ref{SVYDEF} takes $5.531\ast 10^4$ seconds (around $15$ hours), whereas it takes $5.128\ast 10^{-2}$ seconds using our algorithm.
\end{example}
\subsection{Code}\label{code}
The algorithm discussed has been implemented using $\texttt{Sage}$. The reader can find the code at \url{https://github.com/prestontranbarger/NFDSFastComputation}.

\section*{Acknowledgements}
This research was conducted at the $2022$ REU hosted at Texas A$\&$M University and supported by the National Science Foundation (DMS-$2150094$). The authors would like to thank Dr. Matthew Young for his continued support and input throughout the duration of the REU. The authors would like to thank Agniva Dasgupta for his immense help and feedback. The authors would also like to thank Mitch Majure for his input and contributions. 

\bibliographystyle{alpha}
\bibliography{bib}

\end{document}